\documentclass[11pt]{article}
\usepackage{amsmath,amsfonts,amsthm,amssymb,color,bbm}

\usepackage{fullpage}

\newtheorem{thm}{Theorem}[section]
\newtheorem{cor}{Corollary}[section]
\newtheorem{lemma}{Lemma}[section]
\newtheorem{claim}{Claim}[section]
\newtheorem{defn}{Definition}
\newtheorem{conj}{Conjecture}
\newtheorem{rmk}{Remark}
\newtheorem{prop}{Proposition}

\newcommand{\recsum}{\boxplus}
\DeclareMathOperator{\Dis}{Dis}
\DeclareMathOperator{\sign}{sign}

\title{A notion of entropy on the roots of polynomials}
\date{\today}
\author{Aurelien Gribinski\\ 
Princeton University
}

\begin{document}
\maketitle

\begin{abstract}
We introduce a canonical notion of entropy for polynomials analogue to that of random variables in probability. We prove that entropy increases smoothly with respect to finite free addition (see \cite{MSSffp}). In particular we get the new inequality : $ \Dis(p-tp')>\Dis(p) $ for a polynomial $p$,  its derivative $p'$ and t any non zero real. 
\end{abstract}
\section{Introduction}

\subsection{From discriminants to entropy}
 The discriminant of a polynomial is a fundamental quantity in number theory, and yet few inequalities involving discriminants are known that are true for all polynomials with real roots. Most of them are trivial re-writing of the considered quantities. Fix a polynomial of degree $d$, $p(x)= \prod_{i=1}^d (x- \lambda_i)$. Recall the following somewhat canonical definition of the discriminant of $p$, which matches the canonical one for monic polynomials (in all the following we will restrict without loss of generality to monic polynomials):
 \[
 \Dis(p)= \prod_{i < j} (\lambda_i(p)-\lambda_j(p))^2.
 \]
 Notice that defined this way, the discriminant is always positive for real-rooted polynomials with
simple roots.

The discriminant  represents some sort of average pair-wise spacing between all the roots of the polynomial. Therefore, a high  value discriminant amounts to a polynomial with well-spread roots over the real line, whereas a small quantity indicates roots close overall to one another- not well separated from each other (it suffices that two of them are very close to make the discriminant small).In this line of thought, recall that derivating increases in some sense the spreading of the roots. For instance, if a polynomial has a multiple root, derivating  splits the multiple root into
single distinct roots one by one. Other people have studied the effects of  the derivative on the roots of a polynomial in different ways, see \cite{RS} for an overview. In terms of discriminants, it is not possible to compare any normalized version of $\Dis(p)$ with $\Dis(p')$. However, and it constitutes the base case of our result, we can compare the discriminant of a polynomial to which we add (or substract) a scalar multiple of its derivative to the original discriminant. We will prove that:
\begin{thm} \label{basecase}
$\Dis( p-t \alpha p')$ is decreasing for $t<0$ and increasing for $t>0$. In particular, it has a global minimum at zero.
In particular for all $\alpha \in \mathbb{R^{\star}} (\text{set of all real numbers except zero})$,
\[
\Dis(p- \alpha p') > \Dis(p).
\]
\end{thm}
The following result follows by iteration.
\begin{cor}
Let $r(x)=\prod_{i=1}^d(1-\alpha_i x)$, then, if $D$ denotes the derivative operator, and if at least one $\alpha_i$ is nonzero,
\[
\Dis(r(D)p) > \Dis(p).
\]
\end{cor}

Let us investigate how the discriminant is related to probabilistic considerations. \\
Define the set $S_\Lambda=\{Q^t\Lambda Q , Q \in O_d(\mathbb{R}) \}$
where  $ \Lambda=Diag(\lambda_1,....,\lambda_n)$ and  $O_d(\mathbb{R})$ is the group of orthogonal matrices of size $d$. It represents the set of all hermitian matrices whose spectrum is given by the roots of $p$. 
\begin{claim}[Proposition 2.19 from \cite{Zhang}]
The infinitesimal element of volume of  $S_\Lambda$ in $S_d(R)$ is given by 
\[
dV(S_\Lambda)= \prod_{i < j}|\lambda_i-\lambda_j| Vol\big(O_d(\mathbb{R})\big)d\lambda_1...d\lambda_n,
\]
where $Vol$ represents the canonical Lebesgue volume. 
\end{claim}
It justifies the observation that eigenvalues of random matrices tend to repell one another (the likelihood of them being close is small). This result tells us that the size of the discriminant is associated to the size of the infinitesimal volume of matrices with a given spectrum.Therefore we can consider an infinitesimal log-volume of the form $\log \Dis(p)$. It plays the role of an entropy, which is small when there is little randomness that is when the knowledge of the spectrum gives us a lot of information on the underlying symmetric matrix associated to it (matrices with such spectral density form a small ensemble), and which is large when randomness is large, that is on the contrary when the knowledge of the spectrum doesn't give us much information on the matrix(matrices with such spectral density form an ensemble of large infinitesimal volume). It motivates the following definition:
\begin{defn}
For a polynomial of degree $d$, $p:= \prod_{i=1}^d (x- \lambda_i)$, we define its entropy as:
\[
H(p):= \frac{\log[\Dis(p)]}{d(d-1)} = \frac{1}{\binom{d}{2}}\sum_{i,j=1, i<j }^{d}log|\lambda_i-\lambda_j|.
\]
It is by definition $-\infty$ when the polynomial has a multiple root.
\end{defn}

\begin{rmk}
This definition can also be deduced by analogy with free probability. Voiculescu in \cite{Voi} defines the free entropy of a probability measure $\mu$ as:
\[
H(\mu):= \int \int log|s-t| d\mu(s)d\mu(t).
\]
If we consider this averaged logarithm of differences for discrete Dirac measures uniformly distributed on roots of a polynomial ($ \mu_p= \frac{1}{d}\sum_{i=1}^d \delta_{\lambda_i} $)  and that we set aside couples $\{s,t\}$ such that $s=t$ ($log|s-t|$ is not well defined), we retrieve our discrete finite free polynomial entropy. 
\end{rmk}

\subsection{Finite free probability and independent sum for polynomials}
The theory of finite free probability was developed to establish analogues in finite dimensions of tools introduced by Voiculescu in his study of asymptotic random matrices(\cite{Sp}) . 
The following  convolution is a polynomial version of the additive convolution in free probability.
\begin{defn}[See Theorem~2.11 in \cite{MSSffp}]
Consider two  polynomials of degree $d$,  $p= \prod_{i=1}^d (x - \lambda_i)$ and\\
 $q= \prod_{i=1}^d (x - \mu_i)$. Then we define:
\[
p\recsum_d q :=\frac{1}{d!} \sum_{\pi \in \Sigma_d} \prod (x - \lambda_i - \mu_{\pi(i)}),
\]
where  $\Sigma_d$ is the group of permutations of $\{1,\dots,d\}$. It is a degree $d$ polynomial as well. 
\end{defn}

We also express the convolution as a sum of products of derivatives.
\begin{thm} [see Definition~1.1 and Theorem~2.11 in \cite {MSSffp}]\label{diff} We can alternatively write:
\[
p\recsum_d q(x)=  \frac{1}{d!}\sum_{k=0}^d p^{(k)}(x)q^{(d-k)}(0)=\frac{1}{d!}  \sum_{k=0}^d q^{(k)}(x)p^{(d-k)}(0)=\sum_{k=0}^d p^{(k)}(x)(-1)^k\frac{1}{k!} {\sigma_k(q)}
\]
where $p^{(k)}(x)$ is the $k^{th}$ derivative and
\[
 {\sigma_k(q)}= \frac{1}{\binom{d}{k}}\sum_{i_1<i_2...<i_k} \mu_{i_1}\mu_{i_2}\dots \mu_{i_k}.
\]
Through this  formula we can extend the definition to polynomials of lower degree than $d$. 
\end{thm}

\begin{cor}\label{commutederiv}
the $\recsum_d$ operation is symmetric and bilinear, and commutes with derivation.
More explicitly, we have:
\[
D[p\recsum_d q ]= [Dp] \recsum_dq= p \recsum_d [Dq].
\]
As a consequence, for any polynomial $R(D)$ in the differential operator, we have:
\[
R(D)[p\recsum_d q]=  [R(D)p] \recsum_dq= p \recsum_d [R(D)q].
\]
\end{cor}

In particular we have that if  $q(x)=x^d- d\alpha x^{d-1}$,
\[ 
p \recsum_d q= p-\alpha p',
\]
as   $ {\sigma_1(q)}=\alpha$ and all the other symmetric sums are zero; and we retrieve the base derivative case. 

The most remarkable nontrivial property is that this operation preserves realrootedness.
\begin{thm}[see Theorem~1.3 in \cite {MSSffp}]\label{realroot}
If $p$ and $q$ are real-rooted then so is $p\recsum_d q$. 
\end{thm}

This operation is far from being exotic or anecdotic, as the following lemma shows (it was stated in \cite{LR} without the necessary assumption of commutation with the derivation operator).
\begin{lemma}\label{diffcomm}
A linear operator $T: \mathbb{R}^{\leq d} [x] \rightarrow \mathbb{R}^{\leq d}[x] $ ,where $R[x]^{\leq d}$ is the real vector space of
polynomials of degree at most d, is a differential operator which preserves real-rootedness and commutes with the derivation operator if and only if it can be written in the form $T(p)=p \recsum_d q$ for some $q \in \mathbb{R}^{\leq d}[x]$. 
\end{lemma}
\begin{proof} 
The forward direction follows from Theorem~\ref{realroot} for realrootedness and from Theorem~\ref{diff} for commutation with the derivation operator.  For the converse, notice following  Remark~1.15 in \cite{EugeneSo} that a differential operator preserving degree less than $d$ polynomials can always be written as: $T(p)= \sum_{k=1}^d Q_k(x) D^kp$, for polynomials $Q_k$ of degree less than $k$. Commutation with derivation forces (Theorem~5.2 in \cite{EugeneSo}) $Q_k(x)=Q_k(0)$, and we retrieve the polynomial $q$ such that $T(.)=q \recsum_d $ as the only polynomial verifying $\frac{q^{(d-k)}(0)}{d!}=Q_k(0)$ for all $k=0\dots d$. 
\end{proof}

\subsection{Monotonicity of entropy: general result}
We need to introduce at this point the idea of uniform root-dilation (or contraction). For a polynomial $q(x)= \Pi_{i=1}^d\big (x-\lambda_i(q)\big) $, define the root-dilation:
\[
q_t(x)=q(t,x)= \prod_{i=1}^d \big( x- t\lambda_i(q)\big).
\]
Notice that $q_0(x)= x^d$ and $q_1(x)=q(x)$. So we interpolate between the zero polynomial and the original polynomial, using some linear parameter. 
We will prove in Section~3 the following general result (natural generalization of Theorem~\ref{basecase}), which will rely heavily on interlacing tools developed in Section~\ref{2}:
\begin{thm}\label{generalcase}
$\Dis(q_t\recsum_d p)$ is nonicreasing for $t<0$ and nondecreasing for $t>0$. In particular, it has a global minimum at zero.  
For all real-rooted polynomial $q$ of degree $d$ not equal to $(x-c)^d$ (for $c \in \mathbb{R}$), $\Dis(q_t\recsum_d p)$  is in fact decreasing for $t$ close to zero and $t<0$, and increasing  for $t$ close to zero and $t>0$. In particular we get (for $t=1$):
\begin{align*}
\Dis(q \recsum_d p) &>\Dis(p)   &or& &H(q \recsum_d p) &> H(p).
\end{align*}
\end{thm}
\begin{rmk}
We have in the trivial case that
\[
(x-c)^d\recsum_d p= p(x-c),
\]
therefore:
\[
H((x-c)^d\recsum_d p)=H(p).
\]
Also, as roles of $p$ and q are symmetric, if both have distinct roots, we have more generally:
\[
H(q \recsum_d p) > max \big(H(p),H(q)\big).
\]
\end{rmk}

 So doing this additive finite convolution makes the distribution of the roots more "random", as they are more well-spread and less localized. It creates some sort of mixing of the roots. 
As a result, we prove that applying any differential operator which preserves real-rootedness (a necessity to even be able to define the discriminant as a real quantity).  commutes with derivation (Lemma~\ref{diffcomm}), and doesn't decrease the degree (the condition on the polynomial $q$ being that it is of degree $d$), always increases the discriminant - i.e. the average spacing of the roots- which amounts to a sort of smoothness property. 
It corresponds to the known fact from classical probability  that summing independent variables always increases Shannon entropy (or differential entropy). 

On the way (Section~2) we prove new results on polar derivatives (see Definition~\ref{polarderiv}). More specifically we deduce the following (see Corollary~\ref{laguerreimprov}).
\begin{prop}[improved Laguerre's inequality] \label{laguerreprop}
For any real-rooted polynomial $q$ of degree $d$, we have:
\[
(d-1)q'^2(x)- dqq''(x)\geq 0
\]
 for all real $x$ and the inequality is strict when the roots of $q$ are simple. 
\end{prop}

In Section~4, we give an overview of the validity of the entropy perspective (or analogy) in the framework of finite free probability, and we relate it to recent remarks from Terrence Tao(\cite{Tao}).

%%\begin{prop}
%%The entropy is a concave function of its variables, more precisely its hessian is semidefinite negative of rank $d-1$.
%%\end{prop}

%%We can compute the hessian of h straightforwardly: $\nabla^2[h]_{i,i}=-\sum_{j \neq i}\frac{1}{(\lambda_i-\lambda_j)^2}$, and for $i\neq j$, $\nabla^2[h]_{i,j}=\frac{1}{(\lambda_i-\lambda_j)^2}$. We notice that  $|\nabla^2[h]_{i,i}| = \sum_{j\neq i}  |\nabla^2[h]_{i,j}|$, so the matrix is diagonally dominant. It follows from Gershgorin's theorem (the diagonal being negative) that it is semidefinite negative. To see that it is of rank $d-1$, we can look at the upper corner of size $d-1 \cross d-1$ and notice that it is strictly diagonally dominant in the sense that $|\nabla^2[h]_{i,i}| > \sum_{j\neq i}  |\nabla^2[h]_{i,j}|$ for all $i$, and therefore of rank $d-1$. Now, Cauchy interlacing theorem tells us that the eigenvalues of the hessian interlace the eigenvalues of the upper corner, giving us $d-1$ strictly negative eigenvalues. As the vector $(1,1,...,1)$ is trivially in the kernel, we get that the rank is exactly $d-1$.\\

\section{Interlacing tools and polar derivative}\label{2}
We say that two polynomials $p$ and $q$ interlace when their roots alternate, that is if $p= \prod_{i=1}^d (x - \lambda_i)$ and
 $q= \prod_{i=1}^d (x - \mu_i)$, then $\lambda_1\geq \mu_1 \geq  \lambda_2 \geq \mu_2....\geq \lambda_d\geq \mu_d$ or the other way around. We say that the interlacing is strict when inequalities are all strict. 

The main tool to analyze interlacing of polynomials is the Wronskian.
It is defined as:
\[
W[p,q]:=p'q-q'p.
\]
The following easy properties can be found at Section~2.3 in \cite{Wagner}.
\begin{lemma}\label{signw}
If the roots of $p$ and $q$ interlace, then the Wronskian is of constant sign: $W[p,q]  \geq 0 $  or $W[p,q] \leq 0$. Moreover if the interlacing is strict then inequalities become strict. 
\end{lemma}
One of the most simple inequalities that can be interpreted as the positivity of a Wronskian was discovered by Laguerre.
\begin{lemma}\label{Laguerre} (Laguerre's inequality)
If p has simple roots, then we have a strict interlacing between the first and second derivative; more precisely, for all $x \in \mathbb{R}$,
\[
p'^2-p''p=  W[p,p'] (x)>0.
\]
\end{lemma}

 The following theorem can also be found at Section~2.3 in \cite{Wagner}. It shows that interlacing is equivalent to real-rootedness stability with respect to linear combinations. 
\begin{thm}(Hermite-Kakeya-Obreschkoff) \label{HKO}
Let $f$ and $g$ be real polynomials.Then the two following statement are equivalent:\\
1) $f$ and $g$ are real-rooted and their roots interlace (in particular $W[f,g]$ is of constant sign),\\
2) $af+bg$ is real-rooted for all $a,b \in \mathbb{R}$.\\
Moreover, strict interlacing is equivalent to saying  that the roots of $af+bg$ for all $a,b$ are simple. 
\end{thm}
\begin{cor}\label{interlconv}
The finite free addition (convolution) preserves interlacing, or explicitly: if f and g are degree d polynomials that interlace then $p\recsum_d f $ and $p\recsum_d g$ interlace too.
\end{cor}
\begin{proof}
Take $a,b \in \mathbb{R}$. Assume $f$ and $g$ do interlace. It implies by Theorem \ref{HKO} that $af+bg$ is real rooted. But now, using the real-rootedness preservation (Theorem \ref{HKO}) and the linearity properties of $p\recsum_d $, we get that $p \recsum( af +bg)=a( p\recsum f) + b (p\recsum g)$ is real rooted. As $a$ and $b$ are arbitrary, we conclude using Theorem~\ref{HKO} the other way around that $ p\recsum f$ and $p \recsum g$ do interlace.
\end{proof}

We can prove more precisely the following result stated without proof in \cite{LR} (Lemma 4.1):
\begin{cor} \label{convinterlace}
If $W[f,g] (x)\geq 0$ uniformly in $x$, then also $W[f\recsum_d p, g \recsum_d p](x) \geq 0$ for all $x$. (the convolution operation preserves the monotonicity of the interlacing and not merely the interlacing).
\end{cor}

\begin{proof}
Let us assume that $f$ and $g$ are distinct polynomials so that their Wronskian is non identically zero. The function $w(x,t)=W[f\recsum_d p_t, g \recsum_d p_t](x)$ is continuous and is such that $w(x,0) \geq 0$ and $w(x,0)$ goes to infinity when $x$ goes to infinity as it is a nonzero polynomial. For all $t$ between $0$ and $1$, we have following Corollary~\ref{interlconv} either for all $x$, $w(x,t)\geq 0$ or  for all $x$, $w(x,t)\leq 0$. Furthermore assume that for some $t$, $w(x,t)\equiv 0$, that is the identical zero polynomial. Then we would get $f\recsum_d p_t = g \recsum_d p_t$ identically, and by doing the additive convolution with the inverse polynomial (Corollary 6.2 in \cite{M}) of $p_t$, we would get $f=g$ identically, which is contrary to the assumptions. So necessarily for all $t$,  $w(x,t)$ is not the identical zero polynomial and therefore by continuity in $t$ for $x$ asymptotically large, we see that we can't switch from $w(x,t)<0$ for all large $x$ to $w(x,t)> 0$ or vice versa when moving slightly $t$. Therefore for all $t$, $w(x,t) \geq 0$ for all $x$. 
\end{proof}

Let us next introduce the notion of polar derivative, closely related as we will see below to the root-dilation defined above.

\begin{defn}\label{polarderiv}
If $q$ is a polynomial of degree $d$, we call polar derivative the polynomial of degree less than $d-1$ defined by: $\dot{\partial}_x q= dq-xq'$.
\end{defn}

We prove the following result (we were unable to find any reference, Steve Fisk in \cite{Fisk}, Section 4.2 actually claims it is not true but his example is wrong). 

\begin{thm}[Real-rootedness of the polar derivative and interlacing]\label{polar}
For a real-rooted polynomial $q$ of degree $d$ not equal to $x^d$(trivial case), its polar derivative $\dot{\partial}_x q$ of degree $d-1$ or $d-2$ is always real-rooted and its roots interlace the roots of $q'$. Furthermore, $W[\dot{\partial}_x  q, q'] \geq 0$, i.e. the interlacing is positive between the polar derivative and the derivative, and when the roots of $q$ are simple then  $W[\dot{\partial}_x  q, q'] > 0$. 
\end{thm}
\begin{proof}
Let us assume first that $q$ has $d$ distinct roots and let's assume for convenience that it is monic. \\
Denote by $\alpha_{d-1}\dots<\alpha_{1}$ the $d-1$ distinct roots of $q'$.  The sign of $q$ will strictly change from one root of the derivative to another, and using limits at infinity we get $\sign\big(q(\alpha_i)\big)= (-1)^i$, which, as $\dot{\partial}_x q= dq-xq'$, leads to
\begin{equation}\label{sign}
\sign\big( \dot{\partial}_x q(\alpha_i)\big)= (-1)^i.
\end{equation}
The intermediate value theorem gives us, using the alternance of signs from one root to the other, $d-2$  roots $\beta_i$  of  $\dot{\partial}_x q$  such that  $\alpha_{d-1}< \beta_{d-2}< \alpha_{d-2}\dots< \alpha_{2}<\beta_{1}<\alpha_{1}$, that is strictly interlaced with the $\alpha_i$.
Now let us distinguish according to the degree of $\dot{\partial}_x q$. 
Write, for $q=\prod_{i=1}^d\big(x-\lambda_i(q)\big)$: 
\begin{align*} \label{expansion}
q&= x^d + c_{d-1}x^{d-1}+ c_{d-2}x^{d-2}+...\text{(lower degrees)},  \\
\dot{\partial}_x  q&= c_{d-1}x^{d-1}+2c_{d-2}x^{d-2}+...
\end{align*}
where $c_{d-1}=- \sum_{i=1}^d \lambda_i(q)$. \\
If the sum of the roots is zero ( $c_{d-1}=0$), the degree of $\dot{\partial}_x  q$  is at most $d-2$ (it is in fact exactly $d-2$ if $q \neq x^d$) and we already found the $d-2$ roots $\beta_i$ above.\\
 In case $c_{d-1}>0$, we have $\dot{\partial}_x  q(\alpha_{1})<0$ and $\lim_{x \to +\infty} \dot{\partial}_x  q = +\infty$, so we have again by the intermediate value theorem one more root $\beta$ such that $\alpha_{1}<\beta< +\infty$. \\
 In case $c_{d-1}<0$, we have $sign \big(\dot{\partial}_x  q(\alpha_{d-1})\big)= (-1)^{d-1} $  according to Equation~\ref{sign}, and  $\lim_{x \to -\infty} \dot{\partial}_x  q =  (-1)^{d} \infty$. 
 So we have one more root $\beta$ such that $-\infty<\beta< \alpha_{d-1}$ (intermediate value theorem). \\
 Overall, we found all the roots of $\dot{\partial}_xq$ and they interlace with the roots of $q'$. \\
 The case when $q$ has multiple roots is done by a limiting argument: consider $\big(\lambda_1(\epsilon), \lambda_2(\epsilon), \dots, \lambda_d(\epsilon)\big)$ a sequence of distinct roots converging to $\big(\lambda_1, \lambda_2 \dots, \lambda_d\big)$ when $\epsilon \to 0$. $W[\dot{\partial}_x  q_{\epsilon}, q_{\epsilon}'] $ has constant sign for all $\epsilon>0$ by the previous case (simple roots) using Lemma~\ref{signw}, and the Wronskian is a continuous function of the roots therefore in the limit, the constant sign of the Wronskian is preserved and therefore interlacing too (and real-rootedness of $\dot{\partial}_x q$). \\
 Let us prove that in fact $W[\dot{\partial}_x  q, q'] \geq 0$. 
 We already know  that the Wronskian has a constant sign, it is therefore sufficient to prove nonnegativity, and nonnegativity for large $x$ is also sufficient (the sign being constant). We only need to isolate the leading term in $W[\dot{\partial}_x  q,q'](x)$ and show that it is positive. 

After a few straightforward computations we get:
\[
W[\dot{\partial}_x  q,q']= \dot{\partial}_x  q' q'- \dot{\partial}_x  q q''= x^{2(d-2)}\big(c_{d-1}^2(d-1) -2dc_{d-2}\big)+....
\]
Using the fact that $c_{d-1}= -\sum_{i=1}^d \lambda_i(q)$ and  $c_{d-2}=\sum_{i<j} \lambda_i(q)\lambda_j(q)  $ we get that:
\[
c_{d-1}^2(d-1) -2dc_{d-2}= \sum_{i \neq j} (\lambda_i(q)- \lambda_j(q))^2
\]
The positivity is then clear if all the roots are not equal to some constant $c$. If they do, then $W[\dot{\partial}_x  q,q']=0$ identically, and the inequality is also trivially verified.

\end{proof}
We can now prove Proposition~\ref{laguerreprop}.
\begin{cor}[improved Laguerre's inequality]\label{laguerreimprov}
For any real-rooted polynomial $q$ of degree $d$, we have:
\[
(d-1)q'^2(x)- dqq''(x)\geq 0
\]
 for all real $x$ and the inequality is strict when the roots of $q$ are simple. 

\end{cor}
\begin{proof}
It follows from $W[\dot{\partial}_x  q, q'] =(d-1)q'^2- dqq''$. 
\end{proof}

\begin{cor}\label{derivt}
We have:  $\partial_t q_t(x)= \frac{1}{t}\dot{\partial}_x  q_t$. Consequently, for any real-rooted polynomial $q$, $\partial_t q_t(x)$ and $\partial_x q_t$ interlace for all $t \geq 0$. In particular, we prove that  $\partial_t q_t(x)$ is always a real-rooted polynomial as well. 
\end{cor}
\begin{proof}
We recall that $q_t(x)= \prod_{i=1}^d \big(x-t\lambda_i(q)\big)$. We compute easily
\[
\partial_t q_t(x)= \sum_{i=1}^d -\lambda_i(q) \frac{q_t(x)}{x-t\lambda_i(q)}.
\]
On the other hand we have:
\[
\partial_xq_t(x)=  \sum_{i=1}^d \frac{q_t(x)}{x-t\lambda_i(q)}.
\]
So we retrieve the homogeneous polynomial's Euler formula \big($q(t,x)$ is homogeneous in two variables\big):
\[
t\partial_t q_t(x) +x \partial_xq_t(x)= d q_t(x).
\]
Or isolating the scaled polar derivative,
\[
\partial_t q_t(x)= \frac{1}{t}\Big(dq_t(x)-x \partial_xq_t(x)\Big) = \frac{1}{t}\dot{\partial}_x  q_t. 
\]
It then follows from Theorem~\ref{polar} (if $t>0$, the case $t=0$ is trivial as all the roots are at zero).

\end{proof}

\section{Proof of the inequalities}
We will in all the following denote by $r_t(x)= p\recsum_d q_t(x)$ and drop the index $d$ in $\recsum_d$.
We will denote by $\lambda_i(t)$ the roots of $r_t$, according to the context.
\begin{lemma}
\[
\frac{d\lambda_i(t)}{dt}= - \frac{\partial_t r_t}{ \partial_xr_t} \big(\lambda_i(t)\big)
\]
\end{lemma}
\begin{proof}
We can easily prove using the implicit function theorem that the roots are differentiable functions of $t$.
We have the following equality for all time: $ r_t \big( \lambda_i(t)\big) =0$. Therefore,
derivating with respect to $t$ we get
\[
\frac{d\lambda_i(t)}{dt} \partial_xr_t \big(\lambda_i(t) \big)+ \partial_t r_t \big(\lambda_i(t) \big)=0
\]
\end{proof}

\subsection{The base derivative case}
This section is devoted to proving Theorem \ref{basecase}.

\begin{proof}
Let
$q(x):=x^d-\alpha d x^{d-1}$ and: 
\[ r_t(x):=p \recsum q_t= p-t \alpha p'.
\]
We can re-write:
\[
\Dis(r_t)^2= \prod_{i=1}^d\Big(\partial_xr_t\big(\lambda_i(t)\big) \Big)^2.
\]
Now, it is sufficient to show that $f_i(t):=\Big(\partial_xr\big(\lambda_i(t)\big) \Big)^2$ is monotonous  in $t$ for all $i$. We square all the quantities not to bother about the sign but to consider only the absolute value increments.\\
We have:
\[
\frac{df_i(t)}{dt}= 2\partial_xr_t\big(\lambda_i(t)\big) \Big( \frac{d\lambda_i(t)}{dt} \partial_{xx}r_t\big(\lambda_i(t)\big) + \partial_t \partial_xr_t\big(\lambda_i(t) \big) \Big).
\]
On the other hand:
\begin{align}
\partial_t r_t(x)&= -\alpha p'  &  &\text{and}&  \frac{d\lambda_i(t)}{dt}&= -\frac{\partial_t r_t}{\partial_xr_t}\big(\lambda_i(t)\big).
\end{align}

So that:
\begin{align}
\frac{df_i(t)}{dt}&= 2\Big(-\partial_t r_t \partial_{xx}r_t +\partial_t \partial_xr_t\partial_xr_t\Big)\big(\lambda_i(t)\big)\\
&= 2\Big(\alpha p'\big(p''-t\alpha p'''\big) - \alpha p''\big(  p'-t\alpha p''\big)\Big) \big(\lambda_i(t)\big)\\
&= 2t\alpha^2 \Big(p''^2- p'p''' \Big) \big(\lambda_i(t)\big)\\
&=2 t\sigma_1(q)^2 W[p',p''] \big(\lambda_i(t)\big) .
\end{align}
If $t>0$, Lemma~\ref{Laguerre} leads to $\frac{df_i(t)}{dt}>0$ for all $i$ and therefore $\Dis(r_t)^2$ as well as $\Dis(r_t)$ are increasing. If $t$ is negative, they are all decreasing. 
\end{proof}

\subsection{General case for finite free addition}

We now move on to proving Theorem~\ref{generalcase}.
\begin{proof}
We will prove first that for $t>0$, $ \frac{\partial}{\partial t} Dis(r_t(x)) \geq 0$, and $t<0$, $ \frac{\partial}{\partial t} Dis(r_t(x)) \leq 0$. Then, we will prove that for $t$ close to zero the inequalities are strict. The result will immediately follow. We recall that
\[
\Dis(r_t)^2= \prod_{i=1}^d\Big(\partial_xr_t\big(\lambda_i(t)\big) \Big)^2.
\]
Again we will prove that each factor in the former expression does indeed evolve monotonously.
 As in the "base case", let us call $f_i(t):= \partial_xr_t\big(\lambda_i(t)\big)^2$. We have:
 \begin{align*}
 \frac{d f_i(t)}{dt}&=  2\partial_xr_t\big(\lambda_i(t)\big) \Big(\frac{d \lambda_i(t)}{dt} \partial_{xx}r_t\big(\lambda_i(t)\big)+ \partial_t\partial_xr_t\big(\lambda_i(t)\big)\Big)\\
& = 2\Big(-\partial_t r_t \partial_{xx}r_t +\partial_t \partial_xr_t\partial_xr_t\Big)\big(\lambda_i(t)\big)\\
 &= 2W [ \partial_t r_t, \partial_xr_t] \big(\lambda_i(t)\big)\\
 &=  2W\big[  \frac{1}{t}\dot{\partial}_xq_t \recsum_d p , \partial_xq_t\recsum_d p\big] \big(\lambda_i(t)\big),\\
 \end{align*}
 where the last equality follows from the fact that :
 \begin{align*}
  &\partial_t (q_t \recsum_d p)=  \frac{1}{t}\dot{\partial}_xq_t \recsum_d p   \text{ (see Corollary~\ref{derivt})}&  &and &(q_t \recsum_d p)'= q_t' \recsum_d p& \text{ (see Corollary~\ref{commutederiv})}.
  \end{align*}
  Using Theorem~\ref{polar} we get that $W [\dot{\partial}_xq_t  , \partial_xq_t] \big(\lambda_i(t)\big)  \geq 0 $, and using Corollary~\ref{convinterlace}, we get that
 \[
 W [\dot{\partial}_xq_t\recsum_d p, \partial_xq_t\recsum_d p] \big(\lambda_i(t)\big)\geq 0.
 \]
  For $t>0$, we get $ \frac{d f_i(t)}{dt} \geq 0$ and for $t<0$, we get $\frac{d f_i(t)}{dt} \leq 0$.
 What was needed to prove that it is nondecreasing or nonincreasing according to $t$. We will now prove that inequalities are strict for $t$ close to zero.
  To this purpose, we can compute an asymptotic expansion  of $W[ \partial_t q_t \recsum_d p, \partial_xq_t\recsum_d p] $ in $t$. First, using Theorem~\ref{diff} we get that:
  \[
  q_t\recsum_d p= \sum_{k=0}^d p^{(k)}(-1)^k\frac{1}{k!} t^{k} \sigma_k(q).
  \]
  Therefore:
  \begin{align*}
  \partial_t (q_t \recsum_d p)&= - p'\sigma_1(q) +tp''\sigma_2(q) + o(t)   &\text{and}&&   \partial_x (q_t \recsum_d p)&= p'-t p''\sigma_1(q)+ o(t).
  \end{align*}
  
  We then find that $W[ \partial_tr_t, \partial_xr_t] (x)$ is equal to:
   \begin{align*}
   &\big(- p''\sigma_1(q) +tp'''\sigma_2(q)+o(t) \big)\big( p'-t p''\sigma_1(q)+o(t)\big) - \big( p''-t p'''\sigma_1(q)+o(t)\big)  \big(- p'\sigma_1(q) +tp''\sigma_2(q)+o(t)\big)\\
   =& t\big(p'p'''\sigma_2(q)+ p''^2\sigma_1(q)^2- p'''p'\sigma_1(q)^2-p''^2\sigma_2(q) \big) +o(t)\\
   =&t\big(\sigma_1(q)^2-\sigma_2(q)\big)W[p',p''] (x)+ o(t).
   \end{align*}
 A simple computation gives:
 \begin{align*}
  \sigma_1(q)^2- \sigma_2(q) &= \frac{\big(\sum_{i=1}^d{\lambda_i(q)}\big)^2}{d^2}- \frac{\sum_{i<j}2\lambda_i(q)\lambda_j(q)}{d(d-1)}\\
  &= \frac{\big(\sum_{i=1}^d{\lambda_i(q)}\big)^2(d-1) -d \big((\sum_{i=1}^d{\lambda_i(q)})^2- \sum{\lambda_i(q)^2\big)}}{d^2(d-1)}\\
  &= \frac{2}{d(d-1)}\sum_{i<j} (\lambda_i(q)-\lambda_j(q))^2.
  \end{align*}
  We deduce that  if $q$ is not of the form $(x-c)^d$,
 \[
 \sigma_1(q)^2- \sigma_2(q) >0.
 \]

 Using the strict interlacing lemma \ref{Laguerre} again, we get for all $x$ $W[p',p''] (x)>0$ , and for $t>0$ close to zero,
 \[
 W[ \partial_t q_t \recsum_d p, q_t'\recsum_d p] (x)>0
 \]
 and a flipped inequality when $t<0$. It is in particular true for $x=\lambda_i(t)$ and we are done. 
\end{proof}
\section{Conclusion and extensions: information theoretic approach and Hermite polynomials}
The proof of the inequality is quite systematic and elementary, based purely on interlacing of  polynomials and their derivatives. However, this is just the beginning step for an entropic approach of roots of polynomials. It was proven in \cite{M} that finite free addition of polynomials leads to a natural law of large numbers (Theorem~6.5) and a central limit theorem, where sums of random variables are replaced by finite free additions of polynomials with bounded variance. Hermite polynomials are the exact analogue of Gaussians in classical probability as they are the only polynomials with two first cumulants only nonzero (see Section~6.2.2 in \cite{M}).\\
 If we define the expectation of $q$, real-rooted polynomial, as the mean of the roots, $\mathbb{E}(q):=\mu=\sigma_1(q)$, and the variance of $q$,  as $Var(q):= \sigma^2=  \sigma_1(q)^2- \sigma_2(q)$ (see Section~6.2 in \cite{M}), then we can canonically define generalized Hermite polynomials of degree $d$, mean $\mu$ and variance $\sigma^2$:
 \[
 \mathcal{H}_d[\mu,\sigma^2](x):= e^{-\frac{\sigma^2D^2}{2(d-1)}- \mu D} \{x^d\},
 \]
where $D$ is the derivative with respect to $x$ operator. We will denote by $\mathcal{H}_d[\sigma^2]$ the centered usual generalized Hermite polynomials.

 Let us recall that Gaussians are characterized by the following famous entropic property:
\begin{thm}[from \cite{Cover}]
The Gaussian of given variance is the random variable that maximizes differential entropy for random variables with this fixed variance. 
\end{thm}
It was proven in the framework of electrostatics that Hermite polynomials share the same optimal entropic property for our notion of polynomial entropy, which strengthens the analogy. 
\begin{thm}[Theorem~7.6.3 in \cite{Szego}]
$\mathcal{H}_d[\mu,\sigma^2]$, generalized Hermite polynomials with fixed variance $\sigma_2^2$, are the real-rooted polynomials maximizing polynomial entropy (or discriminant) for a given variance. 
\end{thm}
Let us consider the flow obtained when starting at a polynomial $p$ with real roots and adding Hermite polynomials with increasing variance (the equivalent of the perturbation of a random variable with some Gaussian noise, with variance proportional to $t$):
\[
 p(t,x):= p\recsum_d \mathcal{H}_d[2(d-1)t](x)= p\recsum_d [e^{-tD^2}{x^d}]= e^{-tD^2}[p \recsum_d x^d] =  e^{-tD^2}\{p\}(x)= \sum_{k=1}^d \frac{(-1)^kt^k}{k!}D^{2k}p(x),
\]
where the third equality follows from Corollary~\ref{commutederiv}. 
Following Terrence Tao's approach (see \cite{Tao}), we have that (where $D=\partial_x$):
\begin{align*}
\partial_tp(t,x)&=-\partial_{xx} p(t,x)    &\text{and}&&  p(0,x)&=p(x).
\end{align*}
That is, our flow is the reverse heat equation flow starting at $p$. Tao defines the entropy with a different sign convention, and he gets the fact that entropy increases with $t$, a special case of the general result we proved (Theorem~\ref{generalcase}), where we need to change the dilation factor from $t$ to $\sqrt{t}$. In a future paper, we will introduce a notion of Fisher information or energy (the one considered by Tao) that is intimately related to our polynomial entropy. For now let us just state the following conjecture, which strangely resembles Brunn-Minkowski's inequality: 
\begin{conj}[Power entropy inequality, see the free probability analogue in \cite{Voicuentropy}]
We have:
\[
\Dis(p \recsum_d q)^{\frac{1}{\binom{d}{2}}} \geq \Dis(p)^{\frac{1}{\binom{d}{2}}}+ \Dis(q)^{\frac{1}{\binom{d}{2}}},
\]
and there is equality only when $p$ and $q$ are generalized Hermite polynomials .
 It is actually a finite free analogue of the entropy power inequality, as the previous inequality can be read as follows:
\[
e^{2h(p\recsum_d q)}\geq e^{2h(p)}+e^{2h(q)}. 
\]
\end{conj}

This conjecture seems to be true numerically.

\renewcommand{\abstractname}{Acknowledgements}
\begin{abstract}
 I would like to thank my advisor Adam Marcus as well as my friend Masoud Zargar for taking time to read this paper. 
\end{abstract}

\end{document}